\documentclass{amsart}
\usepackage{latexsym}
\usepackage{amssymb}
\usepackage{amsmath}
\usepackage[colorlinks=true, linkcolor=blue, citecolor=blue]{hyperref}

\title{Basics of DTS quasigroups: algebra, geometry and enumeration}
\author{Ale\v s Dr\'apal}
\author{Terry S.~Griggs}
\author{Andrew R.~Kozlik}
\address{Dept.~of Algebra \\ Charles University \\ Sokolovsk\'a 83 \\ 186~75
Praha 8 \\ Czech Rep.}
\address{Dept.~of Mathematics and Statistics \\ The Open University \\
Walton Hall \\ Milton Keynes MK7 6AA \\ United Kingdom}
\address{Dept.~of Algebra \\ Charles University \\ Sokolovsk\'a 83 \\ 186~75
Praha 8 \\ Czech Rep.}
\thanks {Ale\v s Dr\'apal supported by VF20102015006.
Andrew Kozlik supported by SVV-2012-265317.}

\newtheorem{thm}{Theorem}[section]
\newtheorem{lem}[thm]{Lemma}
\newtheorem{prop}[thm]{Proposition}
\newtheorem{cor}[thm]{Corollary}
\theoremstyle{definition}
\newtheorem{ex}[thm]{Example}

\newcommand{\lref}[1]{Lemma~\ref{#1}}
\newcommand{\tref}[1]{Theorem~\ref{#1}}
\newcommand{\pref}[1]{Proposition~\ref{#1}}
\newcommand{\cref}[1]{Corollary~\ref{#1}}
\newcommand{\secref}[1]{Section~\ref{#1}}

\newcommand{\dtr}[1]{\langle\mathtt{#1}\rangle}
\newcommand{\str}[1]{\{\mathtt{#1}\}}

\def\bs{\backslash}
\def\op{^\mathrm{op}}

\keywords{Directed triples system, quasigroup}
\begin{document}

{\Large Electronic version of an article published as \textit{J. Algebra Appl.} \textbf{14} (2015), 1550089, DOI \href{https://doi.org/10.1142/S0219498815500899}{10.1142/S0219498815500899}, \copyright{} copyright World Scientific Publishing Company \url{https://www.worldscientific.com/worldscinet/jaa}.} \pagebreak

\begin{abstract}
A directed triple system can be defined as a decomposition of
a complete digraph to directed triples $\langle x,y,z\rangle$.
By setting $xy =z$, $yz =x$, $xz =y$ and $uu =u$ we get a binary
operation that can be a quasigroup. We give an algebraic
description of such quasigroups, explain how they can
be associated with triangulated pseudosurfaces and report
enumeration results.
\end{abstract}

\maketitle
The notion of a DTS quasigroup is defined in \secref{1}.
In \tref{16} we give an algebraic characterization that is
an important tool in classification and enumeration
of DTS quasigroups. In this respect our main result is
the classification of DTS quasigroups of order 13, where
we found $1\,206\,969$ isomorphism types. Some examples
which may be of particular interest are given in the Appendix.
The method of enumeration is reported in \secref{4}.

DTS loops are those loops that can be obtained from the
(idempotent) DTS quasigroups by prolongation.
\secref{2} explains why
there is little hope that any proper DTS loop will turn
out to be of an algebraic significance. In \secref{3} we show
that DTS quasigroups possess a rich geometrical structure.
This structure offers various invariants, some of which are
exploited in the classification result.

The first paper in which DTS quasigroups were defined is \cite{ldts}.
The connection to this paper is explained below in
\secref{1}.

\section{Directed triple systems and binary operations}\label{1}
Consider a complete directed graph on a set $X$. If $X$ is finite of size
$n$, then it contains
$n(n-1)$ directed edges (arrows). The set of edges can be
decomposed into $n(n-1)/3$ triples if and only if $n\ne 2\bmod 3$.
In such a decomposition each triple has a vertex set of
3, 4, 5 or 6 elements. We shall be considering only the first
alternative. There are four possibilities:
\begin{enumerate}
\item[(1)] $\{(x,y),(y,z),(z,x)\}$ that will be recorded
as $(x,y,z)$ and called a \emph{cyclic} triple,
\item[(2)] $\{(x,y),(y,z),(x,z)\}$ that will be recorded
as $\langle x,y,z\rangle$ and called a \emph{directed} triple,
\item[(3)] $\{(x,y),(y,x),(x,z)\}$ that will be denoted by $\mathbf{i}(x,y,z)$, and
\item[(4)] $\{(x,y),(y,x),(z,x)\}$ that will be denoted by $\mathbf{o}(x,y,z)$.
\end{enumerate}
Triples of types (3) and (4) will be used only in these introductory passages.
Given a decomposition $\mathcal D$ of the complete directed graph on $X$
to triples of types (1), (2), (3) and (4) define
upon $X$ an operation $\cdot$ by setting $a\cdot b = c$ whenever
$\{a,b,c\}$ is the vertex set of
the triple containing the directed edge $(a,b)$. To define
the binary operation $\cdot$ completely put $a\cdot a = a$
for every $a\in Q$ (the operation is idempotent).

Note that $(x,y,z)=(y,z,x)=(z,x,y)$, while $\langle x,y,z\rangle$,
$\langle y,z,x\rangle$ and $\langle z,x,y\rangle$ are pairwise different.
Call $(z,y,x)$ the \emph{opposite} (opposite triple) to $(x,y,z)$,
$\langle z,y,x\rangle$ the opposite to $\langle x,y,z\rangle$,
$\mathbf{i}(z,y,x)$ the opposite to $\mathbf{i}(x,y,z)$ and $\mathbf{o}(z,y,x)$ the opposite to $\mathbf{o}(x,y,z)$.

Call a triple from $\mathcal D$ a \emph{Steiner} triple if the opposite
triple is contained in $\mathcal D$ as well. The Steiner triples will
be denoted by $\{x,y,z\}$ as we shall pay no attention to the way how
$\{x,y,z\}$ is decomposed into the two opposite triples. The reason is that
we shall be interested in $X(\cdot)$ rather than $\mathcal D$. A Steiner
triple induces upon $\{x,y,z\}$ the structure of the (only) idempotent
quasigroup upon this set, and so the decomposition
to the opposites bears no impact upon the definition of $\cdot$.
It is easy to see that the rest of $\mathcal D$ (i.e.~the non-Steiner
triples) can be
derived from the knowledge of the binary operation uniquely.

From here on assume that $\mathcal D$ contains only cyclic and
directed triples.

\begin{lem}\label{11}
Suppose that $x,y \in X$ and $x\ne y$. Then $x\cdot y = y \cdot x$
if and only if $\{x,y,x\cdot y\}$ is a Steiner
triple of $\mathcal D$.
\end{lem}

\begin{proof} Let $z$ be the third vertex of the triple that
contains $(x,y)$. Then $(y,x)$ determines a triple with the vertex set
$\{x,y,z\}$ if and only if $z = y\cdot x$.
\end{proof}

The binary operation will be sometimes replaced by juxtaposition,
with, say, $x\cdot yz$ meaning $x\cdot(y\cdot z)$. A cyclic triple
$(x,y,xy)$ fulfils both $y\cdot xy = x$ and $xy \cdot x = y$.
The latter laws are called \emph{semisymmetric}. If they hold
universally, then they yield a structure of a quasigroup in
which $y\bs x = xy = y/x$ (it is well known and easy to see that
each of the semisymmetric laws implies the other law).

The pair $(X,\mathcal D)$ is called a \emph{Mendelsohn triple system}
(MTS) if all elements of $\mathcal D$ are cyclic. It is clear that
$\mathcal D$ is MTS if and only if the operation $\cdot$ is
semisymmetric (and idempotent, by the definition). Idempotent
semisymmetric quasigroups are thus rightly known as \emph{Mendelsohn
quasigroups}. An MTS is called \emph{pure} if it contains no Steiner
triple.

Commutative semisymmetric quasigroups are called \emph{totally symmetric}
because all their para\-strophes (i.e.~the conjugates) coincide.
Idempotent totally
symmetric quasigroups are also known as  \emph{Steiner} quasigroups
and they are in a
1-to-1 correspondence to \emph{Steiner triple systems}
(STS). An MTS that is not an STS is called \emph{proper}.

In this paper we shall investigate the situation when all elements
of $\mathcal D$ are directed triples. Then $(X,\mathcal D)$ forms
a \emph{directed triple system} (DTS). It is called \emph{pure} if
it contains no Steiner triple. If all triples in $\mathcal D$ are
Steiner, then we get again a Steiner quasigroup, and that happens,
by \lref{11}, if and only if the operation $\cdot$ is commutative.
Call a DTS \emph{proper} if it is not an STS.

The purpose of this paper is to study those DTS for which $X(\cdot)$
is an (idempotent) quasigroup. Note the difference to MTS, where
the semisymmetric law
guarantees that we get a quasigroup structure in all cases.

From here on we shall assume that $\mathcal D$ is a DTS upon $X$.
Our first goal will be to investigate the conditions under which
$X(\cdot)$ is a quasigroup. Such systems will be called
\emph{Latin directed triple systems} (LDTS).
Here and elsewhere there will be a nontrivial intersection
with paper \cite{ldts} where we determined the
existence spectrum of (proper) LDTS. In this paper our approach
is somewhat
different. While \cite{ldts} respects the style
of exposition typical for design theory, here we
concentrate on algebraic and geometrical connections that
are complemented by a report on enumerations of LDTS of
orders up to 13 (the enumeration strategy depends heavily
upon the algebraic model).

\begin{lem}\label{12}
The binary system $X(\cdot)$ is a quasigroup if and only if
it is \emph{divisible} (i.e.~for all $x,y\in X$ there exist
$u,v \in X$ such that $xu = y$ and $vx = y$).
\end{lem}
\begin{proof}
Assume $x \ne y$ and consider the triples that carry $(x,y)$ and
$(y,x)$. Each of the two triples induces three different
ordered triples $(a_1,a_2,a_3)\in X^3$
such that $a_3 = a_1a_2$. There are thus at most six such triples for
which there exist $i,j \in \{1,2,3\}$ with $a_i = x$ and $a_j = y$.
The divisibility condition with respect to $x$ and $y$ means that
such a triple
exists for any choice of $i$ and $j$, $i \ne j$. However, if that is true,
then the triple is determined uniquely since there are exactly six choices
for $(i,j)$.  The divisibility hence implies the uniqueness of divisions.
\end{proof}

Put $Q = X(\cdot)$ and denote by $Q\op$ the binary system with operation
$x*y= yx$. Of course, $Q\op$ is a quasigroup if and only if $Q$ is a quasigroup,
and is induced by the directed triple system
$\mathcal D\op= \{\langle z,y,x\rangle;$
$\langle x,y,z \rangle \in \mathcal D\}$.

\begin{thm}\label{13}
Let $\mathcal D$ be a directed triple system upon a set $X$. Define
a binary operation $\cdot$ on $X$ in such a way that $xy =z$,
$yz=x$ and $xz = y$ whenever $\langle x,y,z \rangle \in \mathcal D$,
and that $xx = x$ for all $x\in X$. Then $X(\cdot)$ is a quasigroup
if and only if for all $\langle x,y,z \rangle \in \mathcal D$ there
exist $x',y',z' \in X$ such that
\[ \langle z',y,x\rangle, \langle z,y',x \rangle, \langle z,y,x'\rangle \in
\mathcal D. \]
In such a case $z' = yx$, $y' = zx$ and $x'=zy$.
\end{thm}
\begin{proof}
Consider $x,y,z \in X$ such that $\langle x,y,z \rangle \in \mathcal D$.
Suppose first that
$\langle y,z',x \rangle \in \mathcal D$ for some $z'\in X$.
Then $z'\ne z$ since $(y,z)$ cannot be covered twice, and so $yz' = x$
implies that $X(\cdot)$ is not a quasigroup. Similarly, we cannot get
a quasigroup if $\langle y,x,z' \rangle \in \mathcal D$ since then
$xz' = y$. We have thus shown that if $X(\cdot)$ is a quasigroup, then
there exists $z'\in X$ with $\langle z',y,x \rangle \in \mathcal D$.
In such a case
\[ z'y = x, \quad z'x = y \quad \text{and} \quad yx = z'.\]
By taking into account that $xy = z$, $yz = x$ and $xz =y$, we see
that the divisibility condition is satisfied with respect to $x$ and $y$.

By turning to $\mathcal D\op$ we get that if $X(\cdot)$ is a quasigroup,
then there exists $x'\in X$ such that $\langle z,y,x' \rangle \in \mathcal D$.
Then
\[zy = x', \quad yx'=z \quad \text{and} \quad zx' = y,\]
and $z$ and $y$ satisfy the divisibility condition.

If there exists $y ' \in X$ with $\langle z,x,y' \rangle \in \mathcal D$,
then $X(\cdot)$ is not a quasigroup by $xy' = z$. We also do not get a
quasigroup if $\langle y',z,x \rangle \in \mathcal D$ since then $y'z = x$.
Hence there exists $y' \in X$ with $\langle z,y',x \rangle \in \mathcal D$
if $X(\cdot)$ is a quasigroup, and then
\[ zy' = x, \quad y'x = z \quad \text{and} \quad zx = y',\]
which supplies the divisibility for $x$ and $z$.

We have seen that the existence of $x',y',z'\in Q$ that satisfy the condition
of the theorem is necessary if $X(\cdot)$ is a quasigroup. We have also
observed that if such
$x'$, $y'$ and $z'$ exist, then the operation $\cdot$ is divisible. That
makes $X(\cdot)$ a quasigroup by \lref{12}.
\end{proof}

\tref{13} thus yields a characterization of LDTS. Our next aim is
to characterize quasigroups $X(\cdot)$ in
terms of the binary operation. Such a quasigroup clearly satisfies
condition (i) of \lref{14}. The lemma
is included to show how the characterization of \tref{16} was discovered.

\begin{lem}\label{14}
Let $Q$ be an idempotent quasigroup. The following properties are equivalent:
\begin{enumerate}
\item[(i)] If $x,y \in Q$, and $a,b \in \{x,y,xy\}$, then
$\{ab,ba\} \cap \{x,y,xy\} \ne \emptyset$.
\item[(ii)] If $x,y \in Q$, then $y \in \{x\cdot xy,xy\cdot x\}$
and $x \in \{xy\cdot y, y \cdot xy\}$.
\item[(iii)] If $x,y \in Q$, then both of the following are true
\begin{enumerate}
\item[(a)] $y =x \cdot xy$ or $y = xy \cdot x$, and
\item[(b)] $y = yx \cdot x$ or $y = x \cdot yx$.
\end{enumerate}
\end{enumerate}
\end{lem}
\begin{proof} Since $Q$ is idempotent, we can consider only the case
$x \ne y$. Then $x$, $y$ and $xy$ are pairwise distinct.
Condition (i) needs a verification only for $\{a,b\} = \{x,xy\}$
and for $\{a,b\} = \{y, xy\}$, and that is exactly the claim of condition (ii).
The first part of (ii) can be expressed by (a), and the second part is (b)
with $x$ and $y$ exchanged.
\end{proof}

\begin{lem}\label{15}
Let $Q$ be a quasigroup such that for all $x,y \in Q$ there holds at least
one of the equalities $x \cdot xy =y = yx \cdot x$ and $ x \cdot yx = y =
xy \cdot x$. Then
\[ x \cdot xy = y \Leftrightarrow y = yx \cdot x
\quad \text{and} \quad x \cdot yx = y \Leftrightarrow xy \cdot x  = y.\]
All four equalities are true if and only if $xy = yx$. If $Q$ is
idempotent, then $xy = yx$ if and only if $\{x,y,xy\}$ is a
subquasigroup.
\end{lem}
\begin{proof} We shall argue by contradiction. There are four possible
violations of our claim. It will suffice to consider just two of them since
the other two follow by a mirror argument.

First, let $x\cdot xy = y$ and $yx \cdot x \ne y$. Then $x \cdot yx = xy \cdot x
=y$. Thus $x \cdot xy = y = x \cdot yx$, and hence $xy = yx$. That yields
$yx \cdot x = xy \cdot x = y$, a contradiction.

Second, let $x\cdot yx = y$ and $xy \cdot x \ne y$. Then $x\cdot xy = yx \cdot x
= y$. Thus $x \cdot yx = x \cdot xy$, $xy = yx$, and $xy \cdot x = yx \cdot
x = y$, a contradiction.
\end{proof}

Call $Q$ a \emph{DTS quasigroup} if $Q$ can be obtained from an LDTS $\mathcal
D$.

\begin{thm}\label{16}
Let $Q$ be an idempotent quasigroup. Then $Q$ is a DTS quasigroup if and only if
all $x,y \in Q$ satisfy
\begin {enumerate}
\item[(i)] $x\cdot xy = y = yx \cdot x$ or $xy \cdot x=y=x\cdot yx$, and
\item[(ii)] $xy \cdot x = y$ implies $xy \cdot y = x$.
\end{enumerate}
\end{thm}
\begin{proof} Let $\mathcal D$ be an LDTS on $X$ such that $X(\cdot)$ is a
quasigroup. Assume that $\langle a,b,c \rangle \in \mathcal D$. By \tref{13}
there exist $a',b',c' \in X$ with $\langle c,b,a'\rangle$, $\langle
c,b',a \rangle$, $\langle c',b,a\rangle \in \mathcal D$. We get the following
table:
\[\begin{array}{|c|c|c|c|c|c|}
\hline
x & y & x\cdot xy & yx \cdot x & xy\cdot x & x \cdot yx \\
\hline
a & b & b & b & b' & ac' \\
a & c & c & c & c' & ab' \\
b & c & c'& a'b  & c  & c \\
\hline
\end{array}.
\]
We see that (i) is obviously true and that (ii) holds if $(x,y) = (b,c)$.
For the other cases of (ii) note that $\{a,b,c\}$ is a Steiner triple
if $b'=b$ or $c'=c$ and that
$xy\cdot x = y$ in every Steiner quasigroup.
Hence (i) and (ii) hold in every
quasigroup that is induced by an LDTS.

Suppose now that (i) and (ii) are true. Define $\mathcal D$
so that $\{x,y,xy\}$ is a Steiner triple if $x\ne y$ are
elements of $X$ such that $xy = yx$. If $xy \ne yx$ let
$(x,y)$ determine the following element of $\mathcal D$:
\begin{enumerate}
\item[(1)] $\langle x,y,xy \rangle$ if $x\cdot xy = yx \cdot x = y$ and
$y \cdot xy = yx \cdot y = x$;
\item[(2)] $\langle x,xy,y \rangle$ if $x \cdot xy = yx \cdot x = y$
and $xy \cdot y = y\cdot yx = x$; and
\item[(3)] $\langle xy,x,y\rangle$ if $xy\cdot x = x \cdot yx =y$
and $xy\cdot y = y \cdot yx = x$.
\end{enumerate}
Every pair $(x,y)$ is covered by a triple from $\mathcal D$. That follows
from our assumption and from \lref{15}. The question is whether two
triples have to agree if they agree in one of the directed edges.
First we shall observe that none of the directed edges that is carried
by a triple determined by (1--3) can appear in a Steiner triple. For
that it is enough to show that any of $x\cdot xy = xy \cdot x$ and $y \cdot xy =
xy \cdot y$ implies $xy = yx$.
That follows from \lref{15}.

Now we shall show that each of conditions (1), (2) and (3) determines
the same set of triples. Assume that $(x,y)$ satisfies (1).
In the next paragraph we shall
observe that then (a) $(x',y') = (x,xy)$ satisfies (2), (b)
$(x'',y'') = (y,xy)$ satisfies (3), and that in both cases we obtain the triple
$\langle x,y,xy \rangle$ again. It follows that a triple determined
by (1) can be determined by (2) and (3) as well. We shall then make
a similar argument starting from (2), and from (3).

By \lref{15} each of conditions (1--3) contains twice more equalities
than needed. When verifying (a) or (b) we shall prove only
one equality for each pair. For (a) note that $x'y'=x\cdot xy = y$,
$x (x\cdot xy) = xy$
and $(x \cdot xy) \cdot xy = y \cdot xy = x$.
For (b) observe that $x'' y''=y\cdot xy = x$,
$ (y \cdot xy) y = xy$ and $ (y \cdot xy)\cdot xy = x \cdot xy = y$.

Assume now (2). We shall show that (a) $(x',y') = (x,xy)$ satisfies (1),
(b) $(x'',y'') = (xy,y)$ satisfies (3), and that both (a) and (b)
yield $\langle x,xy,y \rangle$. We have (a) $x' y' = x \cdot xy = y$,
$x(x\cdot xy) = xy$ and $xy \cdot (x \cdot xy) = xy \cdot y = x$. Furthermore,
(b) $x'' y'' = xy\cdot y = x$, $(xy\cdot y)\cdot xy = x\cdot xy = y$
and $(xy\cdot y) y = xy$.

Finally assume (3). We need to show that (a) $(x',y') =(xy,x)$ satisfies
(1), (b) $(x'',y'') = (xy,y)$ satisfies (2), and that in both cases
we obtain $(xy,x,y)$. Now, (a) $x'y' = xy \cdot x = y$, $xy\cdot(xy \cdot x) =
xy \cdot y = x$ and $x\cdot (xy \cdot x) = xy$, while (b) $x'' y'' =
xy \cdot y = x$, $xy \cdot (xy \cdot y) = xy \cdot x = y$, and
$(xy \cdot y ) y = xy$.

Suppose now that a directed edge $(x,y)$ is covered in two ways. We have
proved that if in one case a Steiner triple is involved, then it is involved
in the other case as well. Since $x$ and $y$ cannot appear in two
different Steiner triples, we can  assume that none of them appears
in a Steiner triple. Thus $xy \ne yx$.

Since each of (1--3) determines the same
set of directed triples we need to consider only
the case when for the given $(x,y)$ there
are true two of conditions (1-3). However, that easily gives $xy = yx$,
a contradiction.
\end{proof}

Laws $x\cdot xy = y$ and $yx \cdot x = y$ are known as the \emph{left}
and \emph{right key laws}, respectively. \tref{16} can be thus rephrased
by saying that DTS quasigroups are those idempotent quasigroups in which
(i) every pair $(x,y)$ is a \emph{key} pair or a \emph{semisymmetric} pair,
and (ii) if $(x,y)$ is semisymmetric, then $(y,x)$ is key. One can ask
what happens when condition (ii) is removed. Then we obtain quasigroups
that can be induced by \emph{hybrid} triple systems \cite{CPR},
i.e.~triple systems which may contain both cyclic and directed
triples. This will be described in detail in a future paper.

\begin{prop}\label{17} Let $Q$ be an idempotent quasigroup. Then $Q$
is a DTS quasigroup if and only if $xy = z$ implies \\
$\phantom{i}
 $ (a)\ $xz =y$ and $yz = x$, \\
$\phantom{i}  \quad \phantom{\alpha}$ or \\
$\phantom{i}
 $ (b)\ $xz =y$ and $zy = x$, \\
$\phantom{i}  \quad \phantom{\alpha}$ or \\
$\phantom{i}
 $ ($c$)\ $zx = y$ and $zy = x$,\\
for all $x,y \in Q$.
\end{prop}
\begin{proof} Note that (a) can be rewritten as $x\cdot xy = y$
and $y\cdot xy = x$. By expressing (b) and (c) in a similar way we
see that the condition of the statement follows from \tref{16} immediately.
Of course, it is also easy to verify it directly from the assumption
that $Q$ is determined by a DTS $\mathcal D$.

To prove the converse we shall start by showing that $x$, $y$ and $z = xy$
form a (commutative) idempotent subquasigroup if at least two
of (a), (b) and (c) can be used for a given pair $(x,y)$.

Suppose first that (a) and (c) apply. Thus $xz=y=zx$ and $yz = zy = x$.
Put  $u=yx$ and consider conditions (a--c) with respect to the pair
$(y,x)$. Then (a,b) give $yu = x$ and (c) gives $uy = x$.
We see that both cases imply $u = z$.

Assume now (a) and (b). Then $xz = y$ and $yz = x = zy$. Put $u =zx$.
It suffices to show that $u =y$ since then the previous case can be used.
Consider $(z,x)$. Then (a,b) $zu = x$ and (c) $uz = x$. Thus $y=u$.

Finally, let (b) and (c) be true. Then $xz = y = zx$ and $zy = x$. It
suffices to show that $yz = x$. Put $u =yz$ and consider $(y,z)$.
Then (a) $zu = y$ and (b,c) $uz=y$.

Let us now define $\mathcal D$. Assume $x\ne y$ and put $z = xy$.
If $\{x,y,z\}$ forms a subquasigroup, take it as a Steiner triple.
If not, include (a) $\langle x,y,z\rangle$, (b) $\langle x,z,y\rangle$,
or (c) $\langle z,x,y\rangle$. We have proved that only one
of these cases applies. It is now clear that each directed edge
is covered by a triple of $\mathcal D$.

Assume (a) $xz = y$ and $yz =x$. Then $(x,z)$ fulfils (b) since
$xy = z$ and $yz = x$, and $(y,z)$ fulfils (c) since $xy =z $
and $xz =y$.

Assume (b) $xz =y$ and $zy = x$. Then $(x,z)$ fulfils (a) since
$xz = y$ and $zy = x$, and $(z,y)$ fulfils (c) since $xz = y$
and $xy =z$.

Assume (c) $zx = y$ and $zy = x$. Then $(z,x)$ fulfils (a) since
$zx =y$ and $xy = z$, and the same equalities imply that
$(z,y)$ fulfils (b).

Therefore any of the three directed edges of a triple from $\mathcal D$ can
be used to induce the triple. Hence a directed edge $(x,y)$ might induce
two different triples of $\mathcal D$ only if at least two of
the alternatives (a--c) apply to $(x,y)$. Above we have proved that then
$\{x,y,xy\}$ forms a Steiner triple. Each directed edge thus induces
only one triple of $\mathcal D$.
\end{proof}

It is true that a shorter proof could be obtained
by uniting \tref{16} and \pref{17} into one statement.
We did not do so for the purpose of future references since
we expect that the characterization of \tref{16} will be mentioned
in the future much more often than the condition of \pref{17}.

\begin{prop}\label{18}
The class of DTS quasigroups is closed under subquasigroups and under
homomorphic images. If both $Q$ and $Q\times Q$ are DTS quasigroups,
then $Q$ is a Steiner quasigroup.
\end{prop}
\begin{proof}
If $Q$ fulfils the condition of \pref{17}, then the condition
is clearly fulfilled both by subquasigroups and by homomorphic images.

Suppose now that $Q$ is a proper DTS quasigroup derived from $\mathcal D$.
Consider $\langle x,y,z \rangle \in \mathcal D$.
Then $(x,y)(y,z)\cdot (y,z) = (z,x)(y,z) = (zy,y)$
equals $(x,y)$ only if $\{x,y,z\}$ is a Steiner triple. However, that
is also true if $(y,z)\cdot (x,y)(y,z) = (y,z)(z,x) = (x,zx)$ equals $(x,y)$.
\end{proof}

\begin{prop} \label{19}
Let $Q$ be a DTS quasigroup. If $Q$ satisfies any of the laws
$x\cdot xy = y$, $yx\cdot x = y$, $xy \cdot x = y$ or $x\cdot yx = y$,
then $Q$ is a Steiner quasigroup.
\end{prop}
\begin{proof} Let  $Q$ be determined by a set of triples $\mathcal D$.
Consider  $\langle x, y, xy\rangle\in \mathcal D$.
By \lref{11} we have to show that $xy = yx$. We have $y\cdot xy = x$ and
so $y\cdot yx = x$ yields $xy = yx$. We also have $x \cdot xy = y$,
and so $xy = yx$ follows from $x \cdot yx = y$. For the other cases
use a mirror argument (or consider $Q\op$).
\end{proof}

There are thus no proper semisymmetric or key DTS quasigroups. However,
there exist many proper flexible DTS quasigroups. Here we refer to
the \emph{flexible} law $x\cdot yx = xy\cdot x$.

\begin{lem}\label{110}
Let $Q$ be a DTS quasigroup determined by $\mathcal D$. Then $Q$
is flexible if $x\cdot yx = xy\cdot x$ for every $\langle x,xy,y\rangle
\in \mathcal D$.
\end{lem}
\begin{proof} We need to show that the restricted assumption
of flexibility implies that $a\cdot ba = ab \cdot a$ for any pair
$(a,b)$, where $a$ and $b$ are distinct elements of $Q$. For that
it clearly suffices to consider the cases $(x,xy)$ and
$(xy,y)$, where $\langle x,xy,y\rangle\in \mathcal D$. The latter
case is immediate since $\langle y, xy, y\cdot xy\rangle \in \mathcal
D$ by \tref{13}, and hence $(xy)(y\cdot xy)= y = x \cdot xy =
(xy \cdot y)(xy)$. For the former case note that
$(x\cdot xy)x = yx$ and that $x(xy\cdot x) = x (x\cdot yx)$ is
equal to $yx$ since by \tref{13} we have $\langle y, yx, x \rangle
\in \mathcal D$ and $\langle x, yx, x\cdot yx \rangle \in \mathcal D$.
\end{proof}

The above lemma can be seen as a variation of \cite[Theorem~2.3]{ldts}.
Note that \cite{ldts} assumes that the set $X$ is finite, while
here we do not exclude the infinite sets. The next statement
corresponds to \cite[Theorem~2.2]{ldts}. It weakens the condition
of \tref{13}, but only for finite sets. Hence we include it without
a proof.

\begin{lem}\label{111} Let $\mathcal D$ be a DTS upon a finite set
$X$. Then $X(\cdot)$ is a quasigroup if and only if
for every $\langle x,y,z\rangle \in \mathcal D$ there exists
$z' \in X$ such that $\langle z',y,x\rangle \in \mathcal D$.
\end{lem}

Let us finish this section by a remark, that an LDTS $\mathcal D$
is pure if and only if the corresponding quasigroup is
\emph{anticommutative} (i.e.~$xy = yx$ implies $x=y$). This follows,
say, from \lref{11}.

\section{From quasigroups to loops}\label{2}

A standard way how to prolong an idempotent quasigroup $Q$
into a loop $Q_1$ consists of adding a (new) neutral element $1$
and setting $x^2 =1$ for all $x \in Q$ (the loop $Q_1$ is
\emph{involutory}).

A loop will be called a \emph{DTS loop} if it can be obtained as
a prolongation of a DTS quasigroup. (Similarly we define
\emph{Steiner} and \emph{Mendelsohn} loops.)

If $x,y \in Q$ are such that $x\cdot xy = y$ (or $yx \cdot x = x$,
or $x\cdot yx = y$ or $xy \cdot x = y$), then the respective
identity holds in $Q_1$ as well, and vice versa. Hence Mendelsohn
loops coincide with semisymmetric loops, and \tref{16} can be
alternatively expressed as:

\begin{thm}\label{21}
A loop $Q_1$ is a DTS loop if and only if for all $x,y \in Q_1$
\begin{enumerate}
\item[(i)] $x\cdot xy = y = yx \cdot x$ or $xy \cdot x=y=x\cdot yx$, and
\item[(ii)] $xy \cdot x = y$ implies $xy \cdot y = x$.
\end{enumerate}
\end{thm}

\begin{prop}\label{22}
A loop $Q_1$ is a DTS loop if and only if $xy = z$ implies \\
$\phantom{i}
 $ (a)\ $xz =y$ and $yz = x$, \\
$\phantom{i}  \quad \phantom{\alpha}$ or \\
$\phantom{i}
 $ (b)\ $xz =y$ and $zy = x$, \\
$\phantom{i}  \quad \phantom{\alpha}$ or \\
$\phantom{i}
 $ ($c$)\ $zx = y$ and $zy = x$,\\
for all $x,y \in Q_1$.
\end{prop}
\begin{proof} Suppose first that $Q_1$ is a prolongation of a DTS
quasigroup $Q$. If $xy = z$ in $Q_1$ and if none of $x$, $y$ and $z$
is equal to $1$, then the implication holds in $Q_1$ because it
holds in $Q$. It is easy to see that it holds as well when
$1 \in \{x,y,z\}$.
On the other hand if $Q_1$ fulfils the implication for all $x,y \in
Q_1$, then $xy = 1$ implies $x=y$. That means that $Q_1$ is involutory
and can be obtained by a prolongation of an idempotent quasigroup
$Q$. If $xy = z$ in $Q$, then either $x=y=z$ or $xy = z$ in $Q_1$.
Hence the implication holds in $Q$ as well and \pref{17} can be used.
\end{proof}

Arguments used in the proof of \pref{18} apply to DTS loops as well, and so
we have:

\begin{prop}\label{23}
The class of DTS loops is closed under subloops and under
homomorphic images. If both $Q_1$ and $Q_1\times Q_1$ are DTS loops,
then $Q_1$ is a Steiner loop.
\end{prop}

A loop that satisfies the law $x\cdot xy = x^2 y$ is called \emph{left
alternative}. The mirror law is the \emph{right alternative} law.

A prolongation $Q_1$ of an idempotent quasigroup $Q$ is left alternative
if and only if $Q$ satisfies the left key law $x\cdot xy = y$.
The prolongation is semisymmetric if and only if $Q$ is semisymmetric.

\begin{prop}\label{24}
Let $Q_1$ be a DTS loop. If $Q_1$ is commutative or left alternative or
right alternative
or semisymmetric, then it is a Steiner loop.
\end{prop}
\begin{proof} If $Q_1$ is commutative, then it is semisymmetric (and
hence also alternative), by \tref{21}. The rest follows from
\pref{19}.
\end{proof}

\begin{lem}\label{25}
Let $Q_1$ be a DTS loop. Suppose that $x,y \in Q_1$ generate a subgroup,
that $1 \notin \{x,y\}$ and that $x \ne y$. Then the subgroup consists
of 1, $x$, $y$ and $xy$. This takes place if and only if $\{x,y,xy\}$
forms a Steiner triple, and that is true if and only if $xy = yx$.
\end{lem}
\begin{proof}
Use \lref{11} if $xy = yx$.
If $\{x,y,xy\}$ forms a Steiner triple, then we clearly get a subgroup.
For the converse it may be assumed that $Q_1$ is a group, by \pref{23}.
The claim follows from \pref{24} since the involutory groups are commutative.
\end{proof}

\begin{prop}\label{26}
Let $Q_1$ be a proper DTS loop. Then it cannot be a (left or right) Bol loop,
or an LC or RC loop, or a Buchsteiner loop or a left or right conjugacy
closed loop.
\end{prop}
\begin{proof} Left Bol loops and LC loops are left alternative. Right Bol
loops and RC loops are right alternative. By \lref{24} we hence need only
to prove that $Q_1$ is commutative if it is a Buchsteiner loop or, say,
a left conjugacy closed (LCC) loop.

LCC loops fulfil the identity $((xy)/x)z = x(y(x\backslash z))$.
Setting $z = 1$ we get $xy = (x\cdot yx) x$ since $Q_1$ is involutory.
Assume that the latter identity holds. Consider the associated
LDTS $\mathcal D$ and assume that $\langle y,x,z \rangle\in \mathcal D$.
Then $(x\cdot yx) x =
xz\cdot x = yx =z =xy$. That makes $\{x,y,z\}$ a Steiner triple,
by \lref{11}, and we see that $Q_1$ is commutative, as required.

In every involutory loop the Buchsteiner law
$x \backslash (xy \cdot z) = (y \cdot zx)/x$
yields $x\backslash (xy \cdot x) = y/x$.
Assume $\langle z,x,y \rangle \in \mathcal D$. Then $z = y/x$
and $xy \cdot x = zx$. Therefore $zx = xz$ and so we get
the commutativity again.
\end{proof}

Proper DTS loops thus never belong to one of the standardly studied
equational classes of loops.

Let $Q_1$ be a loop. The \emph{left nucleus} $N_\lambda$ is formed
by elements $a\in Q_1$  with $a(xy) = (ax)y$ for all
$x,y \in Q_1$.
By shifting $a$ to the right we get the \emph{middle nucleus} $N_\mu$
and the right nucleus $N_\rho$. The \emph{centre} $Z(Q_1)$
consists of all $a \in N_\lambda \cap N_\rho \cap N_\mu$
with $ax = xa$ for every $x \in Q_1$.

Set $C(Q_1) = \{a \in Q_1;$ $ax = xa$ for all $x \in Q_1\}$. By \lref{25},
if $Q_1$ is a DTS loop, then its element $a \ne 1$ belongs to $C(Q_1)$
if and only if $\{a,x,ax\}$ is a Steiner triple for any $x \in Q_1
\setminus \{1,a\}$. Note that $C(Q_1)$ does not have to be a
subloop---below is a counterexample of the smallest order.
For simplicity, we omit commas from the triples.

\begin{ex}\label{counterexample}
Let $X=\{\mathtt{2,3,4,5,6,7,8,9,A,B,C,D,E}\}$ and let $Q_1$ be the DTS loop
determined by the triples
$\str{234}$, $\str{256}$, $\str{278}$, $\str{29A}$, $\str{2BC}$, $\str{2DE}$,
$\str{357}$, $\str{36C}$, $\str{38A}$, $\str{39E}$, $\str{3BD}$,
$\dtr{458}$, $\dtr{469}$, $\dtr{74E}$, $\dtr{76B}$, $\dtr{85D}$, $\dtr{864}$,
$\dtr{954}$, $\dtr{96D}$, $\dtr{97C}$, $\dtr{98B}$, $\dtr{A4C}$, $\dtr{A5B}$,
$\dtr{A6E}$, $\dtr{A7D}$, $\dtr{B47}$, $\dtr{B59}$, $\dtr{B6A}$, $\dtr{B8E}$,
$\dtr{C4D}$, $\dtr{C5E}$, $\dtr{C7A}$, $\dtr{C89}$, $\dtr{D4A}$, $\dtr{D5C}$,
$\dtr{D68}$, $\dtr{D79}$, $\dtr{E4B}$, $\dtr{E5A}$, $\dtr{E67}$, $\dtr{E8C}$.
Then $C(Q_1) = \{1,\mathtt 2,\mathtt 3\}$, but $\mathtt 2 \cdot \mathtt 3 = \mathtt 4 \not\in C(Q_1)$.
\end{ex}

\begin{lem} \label{27}
Let $Q_1$ be a DTS loop. Then $N_\lambda \cup N_\rho
\cup N_\mu \subseteq C(Q_1)$.
\end{lem}
\begin{proof} Suppose that $a,x,y \in Q_1$ are such that $ax = y$.
Let $a$ be first an element of $N_\lambda$. Then
$ay = aa\cdot x = x$ and $a\cdot xy = a(x\cdot ax) = ax \cdot ax =1
= ay \cdot ay = a(y \cdot ay) = a\cdot yx$. Thus $xy = yx$ and
$\{1,a,x,y\}$ is a commutative subgroup of $Q_1$, by \lref{25}.
Hence $ax = xa$.

Let $a$ be now an element of $N_\mu$. Then $ay = aa \cdot x = x$ and
$yy =  1 = xx = x\cdot ay = xa \cdot y$. Thus $xa = y = ax$.
\end{proof}

While the existence spectrum of DTS loops is known, there seem to be
no results that would specify possible sizes of nuclei.

\section {Directed triples and surface triangulations}\label{3}

By a \emph{combinatorial triangulated 2-pseudomanifold}
(shortly \emph{triangulated pseudomanifold}) we shall understand
a finite family $\mathcal F$ of \emph{faces} such that every
face is a three-element set $\{x,y,z\}$ and there exist
unique $x'\ne x$, $y' \ne y$ and $z'\ne z$ with $\{x',y,z\},
\{x,y',z\}, \{x,y,z'\} \in \mathcal F$. In other words,
every edge of $\mathcal F$ is incident to exactly two faces.
Each face determines three edges and three points.
The edges and points yield the \emph{graph} of $\mathcal F$.
The pseudomanifold is said to be \emph{connected} if the
graph is connected. The pseudomanifold is \emph{strongly
connected} if for any two points $x$ and $y$ there exists
a sequences of faces $F_0,\dots,F_k$ such that $F_{i-1}$
and $F_i$ share an edge, $1\le i \le k$, $x$ is incident
to $F_0$ and $y$ is incident to $F_k$. Note that many authors
require (triangulated) pseudomanifolds to be strongly connected.

The main notion we need is that of the triangulated
pseudomanifold as defined above. A more general notion
of \emph{combinatorial 2-pseudomanifolds} (shortly,
\emph{pseudomanifolds}) is defined similarly, but the faces
can be $k$-gons, $k \ge 3$. Taken formally, the \emph{face}
is then a pair $\{(y_1,\dots,y_k),(y_k,\dots,y_1)\}$, where
$y_1,\dots,y_k$ are pairwise distinct points
and $(y_1,\dots,y_k)$ is regarded as a cyclic sequence.
By choosing one element
of the pair we choose an orientation of the face. A combinatorial
2-pseudomanifold is \emph{orientable} if the orientation can
be fixed in such a way that two different faces that
share an edge induce upon the edge opposite orientations.
If such a coherent orientation is given, we speak about an
\emph{oriented pseudomanifold}. An oriented pseudomanifold
can be considered as a family of \emph{oriented faces}
$(y_1,\dots,y_k)$. Orientable pseudomanifolds will be
called here \emph{(combinatorial) pseudosurfaces}.

Let $\mathcal D$ be a finite DTS.
Elements $\langle x,y,z \rangle \in \mathcal D$ that
do not yield a Steiner triple will be called \emph{unidirectional}.
Denote by $\mathcal F$ the set of all $\{x,y,z\}$,
where $\langle x,y,z \rangle$ runs through all unidirectional
triples of $\mathcal D$. Consider now $\mathcal F$ as a set
of faces. Each edge $\{x,y\}$ is incident to two faces, and
so we get a pseudomanifold. In general, the pseudomanifold
does not have to be orientable.

Suppose now that $\mathcal D$ is a finite LDTS. Orient
$\{x,y,z\}\in \mathcal F$ as $(x,y,z)$ if $\langle x,y,z
\rangle \in \mathcal D$. It follows from \tref{13} that
this defines a coherent orientation. Hence $\mathcal F$
is a pseudosurface. We shall call it the pseudosurface
of $\mathcal D$ (or of $Q$ if $Q$ is the DTS quasigroup
that determines $\mathcal D$).

Consider $\langle y_0,x,y_1\rangle \in \mathcal D$. There
exist $k \ge 2$ and points
$y_0,y_1,y_2,\dots,y_k$ that are pairwise distinct such that
$\langle y_1,x,y_2\rangle, \dots, \langle y_k, x, y_0
\rangle \in \mathcal D$. Call $(y_0,y_1,\dots,y_k)$
an (oriented) \emph{residual face}. The triangular
faces $\{y_0,x,y_1\},\dots,\{y_k,x,y_0\}$ form its
\emph{cap}. The oriented residual face $(y_0,\dots,y_k)$
is said to be
\emph{singular} if $(y_k,\dots,y_0)$ is an oriented residual
face as well. To see how singular residual faces relate
to flexibility we need the following lemma. It analyzes
the situation when two residual faces share an edge.

\begin{lem}\label{31}
Suppose that $\mathcal D$ contains $\langle y_0,x_1,y_1 \rangle$,
$\langle y_1, x_2, y_0\rangle$, $\langle y_1, x_1, y_2 \rangle$
and $\langle y_2',x_2,y_1\rangle$. Then $y_2' = y_2$ if and only
if $y_1 \cdot y_0y_1 = y_1y_0 \cdot y_1$.
\end{lem}
\begin{proof} By our assumptions $x_1 = y_0y_1$, $x_2 = y_1y_0$,
$y_2 = y_1x_1 = y_1 \cdot y_0y_1$ and $y_2' = x_2y_1 = y_1y_0
\cdot y_1$.
\end{proof}

\begin{cor}\label{32}
A finite DTS quasigroup $Q$ is flexible if and only if all
residual faces of $Q$ are singular.
\end{cor}
\begin{proof} Combine \lref{31} with \lref{110}.
\end{proof}

Denote by $O_k$ a $k$-gonal bipyramid, i.e.\ a graph of $k+2$ vertices with
a cycle of length $k\ge 3$, in which the remaining two vertices
are connected to the elements of the cycle (the graph
contains $3k$ edges). \cref{32} immediately yields:

\begin{thm}\label{33} A flexible DTS quasigroup of order $n$
exists if and only if the complete graph $K_n$ can be decomposed
to triangles and graphs $O_k$, $k \ge 3$.
\end{thm}

Note that the number of nonisomorphic flexible quasigroups
of order $n$ can be much bigger than the number of nonisomorphic
decompositions of $K_n$, as each $O_k$ can be oriented in
two ways (if $k = 4$ then there are, in addition,
three ways how to choose the non-oriented residual face).

The existence spectrum of odd order flexible DTS quasigroups was
determined in \cite[Theorem~4.4]{ldts}. The even case is being investigated.

When we put aside the singular residual faces we get
a set of oriented faces that yields an oriented pseudosurface.
We call it the \emph{residual pseudosurface}. It is
obtained from the pseudosurface of $Q$ by cutting away the caps.

The proof that there are no DTS quasigroups of order 10
\cite[Theorem~3.3]{ldts} is based upon showing that the parameters
of a potential residual pseudosurface induce a surface with parameters
that would violate the parity of the Euler characteristic.

The notion of the strong connectivity can be used to partition
the pseudosurface of a DTS quasigroup $Q$ into \emph{components}.
Each component
possesses a genus, and the list of genera can be considered
as an invariant of $Q$. The components induced by
a singular residual face are called \emph{flexible}.
Their graph is isomorphic to $O_k$ for some $k \ge 3$.

Note however that a component may
still be a proper pseudosurface, i.e.~it does
not have to be a (combinatorial) surface
(a formal definition of a surface can be found below).
This fact seems
to make the geometrical approach a less potent tool
than might be expected when
proving the existence or non-existence of
DTS quasigroups of orders greater than 10.

Nevertheless, the gained geometrical insight naturally leads
to a construction that uses latin bitrades
to diminish the number of Steiner triples in a
DTS quasigroup (in particular, to build a proper DTS quasigroup
from a Steiner quasigroup).

By a \emph{latin bitrade $T$} we shall understand a pair
$T= (L,R)$ where $L$ and $R$ are two disjoint sets
consisting of ordered triples such that if $1\le i < j \le 3$,
and $a=(a_1,a_2,a_3) \in L$, then $\{a_i,a_j\}$ determines
the triple $a$ uniquely, $a_i \ne a_j$,
and there exists $b=(b_1,b_2,b_3)\in R$
with $(a_i,a_j) = (b_i,b_j)$. The meaning of the \emph{mates}
$L$ and $R$ is interchangeable, and thus for $R$ there
apply symmetric conditions.

Our definition of latin bitrades is tailored to present needs. Instead
of requiring that $a_i \ne a_j$ and that $\{a_i,a_j\}$
determines the triple $a$ it is usual to require only
that $(a_i,a_j)$ determines $a$. Another, a more restrictive
definition, includes a condition that $a_i \ne a'_j$
for all $(a_1',a_2',a_3')\in L$, $1\le i < j \le 3$.
These variations have no structural impact and can be
solved by renaming of elements.

Note that by considering the family of all $\{a_1,a_2,a_3\}$
and $\{b_1,b_2,b_3\}$, where $(a_1,a_2,a_3) \in L$ and
$(b_1,b_2,b_3) \in R$, we get a pseudomanifold. By
choosing reverse orientations for elements of $L$ and $R$
we see that the pseudomanifold is orientable (it is a
pseudosurface).

\begin{prop}\label{34}
Let $(L,R)$ be a Latin bitrade and let $\mathcal D$ be an LDTS
such that $\{a_1,a_2,a_3\}$ is a Steiner triple in $\mathcal D$ for every
$(a_1,a_2,a_3) \in L$. Change $\mathcal D$ into $\mathcal D'$
in such a way that these Steiner
triples are replaced by directed triples $\langle a_1,a_2,a_3\rangle$
and $\langle b_3,b_2,b_1\rangle$, where $(a_1,a_2,a_3)\in L$ and
$(b_1,b_2,b_3) \in R$. Then $\mathcal D'$ is an LDTS as well.
\end{prop}
\begin{proof} Suppose that $(b_1,b_2,b_3)\in R$ is chosen in such a way
that $b_1 = a_1$ and $b_3 = a_3$ where $(a_1,a_2,a_3) \in L$. Then
$\langle a_1,a_2,a_3 \rangle$ covers $(a_1,a_3)$ and $\langle
b_3,b_2,b_1\rangle$ covers $(a_3,a_1)$. By treating cases
$(b_1,b_2) = (a_1,a_2)$ and $(b_2,b_3) = (a_2,a_3)$ in a similar
way we see that \tref{13} can be used.
\end{proof}

If $Q'$ is the quasigroup determined by $\mathcal D'$, and $Q$ is
determined by $\mathcal D$, then we shall say that \emph{$Q'$ is derived
from $Q$ by means of a latin bitrade $(L,R)$.}

By a \emph{surface} we understand here a strongly connected
pseudosurface in which all faces incident to a point rotate
around the point. To turn a strongly connected pseudosurface
into a surface it suffices to divide a point into several
new points (let us call them \emph{vertices}) so that each vertex
corresponds to a cycle of faces around the point. If the
pseudosurface is triangulated, then such a cycle around a point
$x$ takes form $\{y_0,x,y_1\},
\{y_1,x,y_2\},\dots,\{y_k,x,y_0\}$.
A pseudosurface is thus a surface if and only if for each
point $x$ there is only one such cycle.

A DTS quasigroup $Q$
yields components that are pseudosurfaces, and each such pseudosurface
yields a surface by the procedure we have just described. We shall
speak about a \emph{surface constituent} of $Q$.
If a vertex corresponds to the cap of a residual face, it will be
referred to as a \emph{middle} vertex, otherwise it will be referred to
as a \emph{residual} vertex.

\begin{prop} \label{35}
A DTS quasigroup $Q$ can be derived by means of a latin
bitrade from a Steiner quasigroup if and only if each surface constituent
of $Q$ is vertex 3-colourable.
\end{prop}
\begin{proof}
In a vertex 3-colourable triangulated surface with a chosen coherent orientation
the faces can be divided
into two classes according to the cyclic ordering of the vertex classes
that is induced by the orientation of the face. The surface is hence
face 2-colourable. For each face colour consider the set of
ordered triples
$(a_1,a_2,a_3)$ such that $\{a_1,a_2,a_3\}$ is a face of the
given colour and $a_i$ is a vertex of colour $i$. It is clear
that the obtained
sets are mates of a latin bitrade.

Assume that all surface constituents of $Q$ are vertex 3-colourable.
Each constituent thus defines a latin bitrade. The identifications
of vertices that are needed to turn the surface constituent into
the corresponding (pseudosurface) component can be carried out
in the bitrade structure without violating the definition
of the latin bitrade. Furthermore,
the obtained latin bitrades can be aggregated
into one bitrade, and this bitrade determines a Steiner
quasigroup from which $Q$ can be derived.

If $Q$ was derived from a Steiner quasigroup, then the used
latin bitrade can be interpreted as a pseudosurface. The
obtained pseudosurface coincides
with the pseudosurface of $Q$. Each constituent of $Q$ can be thus
interpreted as a latin bitrade in which the projections along
the 1st, 2nd and 3rd coordinate yield three sets that are
pairwise disjoint. These sets yield the three colours of vertices.
\end{proof}

Each nonflexible component of a DTS quasigroup $Q$ yields in an obvious
way a \emph{residual component} and a \emph{residual constituent}.
Note that a surface constituent is vertex 3-colourable if and only
if the graph of its residual constituent is bipartite.

It is well known that
triangulated surfaces of genus 0 (the spherical surfaces) are
vertex 3-colourable if and only if they are Eulerian (i.e.~if
each vertex
is of an even degree). Using \tref{33} we see that a flexible
DTS quasigroup can be derived from a Steiner quasigroup by
means of latin bitrades if and only if each component corresponds
to $O_k$ for an even $k=2m$. The trades involved in such derivation
of flexible DTS quasigroups possess a transparent
structure. They are sometimes called \emph{bicyclic} and
can be represented by $L= \{(x_1,y,x_2)$, $(x_2,z,x_3)$,
$\dots$, $(x_{2m-1},y,x_{2m})$, $(x_{2m},z,x_1)\}$ and by $R$
that is
obtained from $L$ by exchanging all occurrences of $y$ and $z$.
Note that by permuting, say, the first and second coordinate we
get a latin bitrade that can be used to build a DTS quasigroup
as well. However, the resulting quasigroup will not be flexible
if $m \ge 3$.

If $m=2$, then the STS of the initial Steiner quasigroup
contains $\{x_1,y,x_2\}$, $\{x_3,y,x_4\}$, $\{x_1,z,x_3\}$
and $\{x_2,z,x_4\}$. This is known as a \emph{Pasch configuration}.
Its transformation via the corresponding latin bitrade is used
in \cite{ldts} several times (e.g.~in Proposition~4.1).

\section{Enumeration and classification}\label{4}
To enumerate DTS quasigroups we use the program Mace4~\cite{mace} which
is part of the package Prover9, an automated theorem prover for
first-order and equational logic. While Prover9 searches for a proof,
Mace4 is generally used to search for finite counterexamples, however
it can also be used to enumerate all structures of some finite order
that satisfy a given set of equations. For example, in order to generate
all proper DTS quasigroups of order 7 we provide Mace4 with the
following input\\
\begin{tabular}{rl}
\scriptsize  1 & \verb:assign(max_models, -1).:\\
\scriptsize  2 & \verb:assign(domain_size, 7).:\\
\scriptsize  3 & \verb:formulas(sos).:\\
\scriptsize  4 & \verb:    x * y = x * z  ->  y = z.:\\
\scriptsize  5 & \verb:    y * x = z * x  ->  y = z.:\\
\scriptsize  6 & \verb:    x * x = x.:\\
\scriptsize  7 & \verb:    (x * (x * y) = y  &  (y * x) * x = y) |:\\
\scriptsize  8 & \verb:        ((x * y) * x = y  &  x * (y * x) = y).:\\
\scriptsize  9 & \verb:    (x * y) * x = y  ->  (x * y) * y = x.:\\
\scriptsize 10 & \verb:    0 * 1 != 1 * 0.:\\
\scriptsize 11 & \verb:end_of_list.:\\
\end{tabular}\\
The equations on lines 7, 8 and~9 correspond to the characterisation of
DTS quasigroups given in \tref{16}. Mace4 tends to generate the results
faster using this characterisation than if the characterisation from
\pref{17} is used. When enumerating proper DTS quasigroups of order~12
it runs approximately 20~times faster. On the right side of the
implication on line~9 either one of the key laws or a conjunction of the
key laws can be used. Similarly the left side of the implication can be
replaced with \,\verb:x * (y * x) = y:\, or with a disjunction of the two
expressions. As one might expect, using the disjunction on the left
gives the worst running time of all. The remaining six possible
combinations all do equally well.

Mace4 can instantly enumerate the DTS quasigroups of orders up to~9 and
determine that none exist for orders 4, 6 or~10. The enumeration of
DTS quasigroups of order~12 can be achieved in a matter of minutes.

The smallest proper DTS quasigroup is of order~7. It is unique up to
isomorphism and yields a single surface constituent which is isomorphic
to~$O_4$.

For proper DTS quasigroups of order~9 there exist three isomorphism
types. The first two types each yield a single surface constituent
isomorphic to~$O_6$, however one of these is flexible while the other
is not, i.e.\ their residual constituents are non-isomorphic. The third
type yields a surface constituent of genus~$1$ consisting of 3 residual
faces.

For proper DTS quasigroups of order~12 there exist two isomorphism
types. Their pseudosurfaces differ only in orientation. Each type
yields three residual surface constituents, all isomorphic to a
tetrahedron.

All DTS quasigroups of order up to 12 are explicitly described in~\cite{ldts}.

At order~13 the combinatorial explosion takes over. If we attempt to
generate the DTS quasigroups of order~13 using the above input, Mace4
soon runs out of memory. In comparison for Steiner triple systems the
combinatorial explosion takes place at order~19~\cite{kaski}.

We split the task of enumerating DTS quasigroups of order~13 into more
manageable tasks by placing restrictions on the degrees of middle
vertices (cf.\ Section~\ref{3}). We first focused on generating the DTS quasigroups with middle
vertices of degree at most~6, then we focused on generating those that
contain at least one middle vertex of degree greater than~6. Thus the
task was split into generating proper DTS quasigroups of order~13 such
that
\begin{enumerate}
\item[1.] all middle vertices have degree~3;

\item[2.] all middle vertices have degree at most~4 and
there exists a middle vertex of degree~4;

\item[3.] all middle vertices have degree at most~6,
there exists a middle vertex of degree~5 and there may or may not
exist a vertex of degree~6;

\item[4.] all middle vertices have degree at most~6 and
there exists a middle vertex of degree~6 but no vertex of degree~5;

\item[5.] there exists a middle vertex of degree~7;

\item[6.] there exists a middle vertex of degree~8;

\item[7.] there exists a middle vertex of degree~9;

\item[8.] there exists a middle vertex of degree~10;

\item[9.] there exists a middle vertex of degree~12.
\end{enumerate}

Mace4 generated a total of 16\,682\,471 quasigroups in 59.2 hours on a
computer equipped with an Intel Xeon E5620 $2.40\,\mathrm{GHz}$ CPU with
$12\,\mathrm{MB}$ of cache. This does not include the time needed to remove
the isomorphic quasigroups. Details are given in Table~\ref{tbl:stats}.

\begin{table}
\begin{tabular}{rrrr}
Task & Generated & Isomorphism types & Time to generate\\
\hline
      1 &           12 &          1 &    2 minutes\\
      2 &     217\,292 &     8\,004 & 24.5 hours\\
      3 &     831\,487 &   106\,446 &  4.0 hours\\
      4 &  1\,337\,912 &    87\,019 & 14.2 hours\\
      5 &  1\,960\,056 &   258\,251 &  2.0 hours\\
    6.1 &  3\,368\,344 &   353\,637 &  3.2 hours\\
    6.2 &  1\,090\,528 &    34\,079 &  2.4 hours\\
6.3 (a) &  1\,327\,664 &    91\,738 &  1.3 hours\\
6.3 (b) &     686\,064 &   299\,641 &  0.6 hours\\
      7 &     325\,644 &    36\,184 &  0.9 hours\\
      8 &  4\,779\,308 &   401\,683 &  3.7 hours\\
      9 &     758\,160 &    63\,180 &  2.2 hours\\
\hline
 Total & 16\,682\,471 &1\,206\,967 & 59.2 hours
\end{tabular}
\caption{The number of proper DTS quasigroups of order~13 generated by Mace4 in each task of the enumeration.}
\label{tbl:stats}
\end{table}

When dealing with the DTS quasigroups that have a middle vertex of
degree~8, Mace4 ran out of memory. The task was split further as
follows. Denote the point corresponding to the middle vertex of degree~8
as~$\mathtt 0$, the corresponding residual face as
$(\mathtt 1,\mathtt 2,\ldots,\mathtt 8)$ and the remaining points as
$\mathtt 9$, $\mathtt T$, $\mathtt E$ and~$\mathtt W$. We split the task
based on how these four remaining points relate to the point~$\mathtt 0$.
There are three possibilities, one of which had to be split further
because Mace4 ran out of memory.
\begin{enumerate}
\item[6.1] The remaining points form two Steiner triangles with the
point~$\mathtt 0$, e.g.\ $\{\mathtt 0,\mathtt 9,\mathtt T\}$ and
$\{\mathtt 0,\mathtt E,\mathtt W\}$;

\item[6.2] there exists another middle vertex corresponding to the
point~$\mathtt 0$ and the remaining four points correspond to vertices which
form a cycle around this middle vertex, e.g.\ the LDTS contains the
directed triples
$\langle\mathtt 9,\mathtt 0,\mathtt T\rangle$,
$\langle\mathtt T,\mathtt 0,\mathtt E\rangle$,
$\langle\mathtt E,\mathtt 0,\mathtt W\rangle$ and
$\langle\mathtt W,\mathtt 0,\mathtt 9\rangle$; or

\item[6.3] there exists a residual vertex corresponding to the point~$\mathtt 0$
and the remaining four points correspond to vertices which form a cycle
around this residual vertex, e.g.\ the LDTS contains the directed triples
$\langle\mathtt 0,\mathtt 9,\mathtt T \rangle$,
$\langle\mathtt T,\mathtt E,\mathtt 0 \rangle$,
$\langle\mathtt 0,\mathtt E,\mathtt W \rangle$ and
$\langle\mathtt W,\mathtt 9,\mathtt 0 \rangle$, and further
  \begin{enumerate}
  \item[(a)] $\mathtt 9 \cdot \mathtt W = \mathtt T$ or

  \item[(b)] $\mathtt 9 \cdot \mathtt W$ is one of the points
  $\mathtt 1,\ldots,\mathtt 8$.
  \end{enumerate}
\end{enumerate}
When dealing with the case of the two Steiner triangles $\{\mathtt 0,\mathtt 9,\mathtt T\}$
and $\{\mathtt 0,\mathtt E,\mathtt W\}$ above, $\mathtt T \cdot\mathtt W$ must be one of the points
$\mathtt 1,\ldots,\mathtt 8$. Assigning $\mathtt T \cdot\mathtt W =\mathtt 1$ reduces the number of
isomorphic models generated and was necessary to prevent Mace4 from
running out of memory. Similarly in 6.3~(b) we assign $\mathtt 9 \cdot\mathtt W =\mathtt 1$.

After putting all the results together we found $1\,206\,969$
isomorphism types of DTS quasigroups of order~13. Out of these $8\,444$
are pure and $924$ are flexible (including the 2 Steiner quasigroups).
There do not exist any pure flexible DTS quasigroups of order~13.

To remove the isomorphic models, the results were first split into
smaller classes according to an invariant which is derived from how
each point of the pseudosurface splits into vertices of the surface,
taking into account the degree of each vertex and whether it is a
middle vertex or a residual vertex. Isomorphic models were then removed
from each class using a custom program which exploits the geometric
structure of DTS quasigroups to find possible isomorphisms. Afterwards,
each of these classes was checked using the GAP~\cite{gap} package
LOOPS~\cite{loops} to confirm that its contents are indeed pairwise
non-isomorphic.

The isomorphism types were then classified according to the genera of
their surface constituents and according to their automorphism group,
see Tables~\ref{tbl:genera} and~\ref{tbl:automorphisms}.
Table~\ref{tbl:genera} also gives the number of non-isomorphic
pseudosurfaces yielded by the DTS quasigroups in each class.
For example the last line in Table~\ref{tbl:genera} indicates that there
exist exactly 6 non-isomorphic DTS quasigroups of order~13 that consist
of 2 surface constituents of genus~1
(see Example~\ref{ldts13g11}).
These 6 quasigroups yield only 2 non-isomorphic pseudosurfaces.
The number of non-isomorphic pseudosurfaces in each class was determined
using \texttt{shortg} from the package \textbf{nauty}~\cite{nauty}. The
automorphism groups in Table~\ref{tbl:automorphisms} were determined
using GAP. We refer to the dihedral group of order~$2n$ as~$D_{2n}$.

\begin{table}
\begin{tabular}{cccccrr}
&&&&&
\multicolumn{1}{c}{Number of}&
\multicolumn{1}{c}{Number of}\\
\multicolumn{5}{c}{Number of surface constituents of genus~$g$}&
\multicolumn{1}{c}{non-isomorphic}&
\multicolumn{1}{c}{non-isomorphic}\\
$\;g=0\;$ & $\;g=1\;$ & $\;g=2\;$ & $\;g=3\;$ & $\;g=4\;$ &
\multicolumn{1}{c}{quasigroups} &
\multicolumn{1}{c}{pseudosurfaces}\\
\hline
0 & 1 & 0 & 0 & 0 &   392\,685 & 189\,280\\
1 & 0 & 0 & 0 & 0 &   391\,805 & 166\,149\\
2 & 0 & 0 & 0 & 0 &   152\,818 &  26\,227\\
0 & 0 & 1 & 0 & 0 &   117\,368 &  58\,588\\
1 & 1 & 0 & 0 & 0 &    80\,875 &  16\,100\\
3 & 0 & 0 & 0 & 0 &    32\,100 &   2\,098\\
1 & 0 & 1 & 0 & 0 &    14\,019 &   3\,162\\
0 & 0 & 0 & 1 & 0 &    10\,636 &   5\,374\\
4 & 0 & 0 & 0 & 0 &     6\,000 &      267\\
2 & 1 & 0 & 0 & 0 &     5\,896 &      505\\
5 & 0 & 0 & 0 & 0 &        955 &       28\\
1 & 0 & 0 & 1 & 0 &        769 &      189\\
3 & 1 & 0 & 0 & 0 &        533 &       36\\
0 & 0 & 0 & 0 & 1 &        246 &      131\\
2 & 0 & 1 & 0 & 0 &        178 &       18\\
4 & 1 & 0 & 0 & 0 &         40 &        3\\
6 & 0 & 0 & 0 & 0 &         24 &        3\\
1 & 0 & 0 & 0 & 1 &         14 &        4\\
0 & 2 & 0 & 0 & 0 &          6 &        2\\
\hline
 &  &  &  & Total &1\,206\,967 & 468\,164\\[0.7ex]
\end{tabular}
\caption{Classification of the isomorphism types of proper
DTS quasigroups of order~13 according to the genera of their surface
constituents.}
\label{tbl:genera}
\end{table}

\begin{table}
\begin{tabular}{crrr}
                & Number   &      & \\
Aut($Q$)        & of types & Pure & Flexible \\
\hline
$C_1$       & 1\,202\,669 & 8\,406 & 864 \\ %[ 1, 1 ]
$C_2$       &      4\,163 &     36 &  43 \\ %[ 2, 1 ]
$C_3$       &          92 &      0 &   8 \\ %[ 3, 1 ]
$C_2 \times C_2$&      17 &      0 &   0 \\ %[ 4, 2 ]
$C_5$       &           8 &      0 &   0 \\ %[ 5, 1 ]
$S_3$       &           7 &      0 &   1 \\ %[ 6, 1 ] the flexible one is STS
$C_6$       &           5 &      0 &   4 \\ %[ 6, 2 ]
$C_{10}$    &           2 &      0 &   2 \\ %[ 10, 2 ]
$D_{10}$    &           1 &      0 &   1 \\ %[ 10, 1 ]
$D_{12}$    &           2 &      0 &   0 \\ %[ 12, 4 ]
$C_{13}$    &           2 &      2 &   0 \\ %[ 13, 1 ]
$C_{13} \rtimes C_3$ &  1 &      0 &   1 \\ %[ 39, 1 ] STS
\hline
Total       & 1\,206\,969 & 8\,444 & 924\\[0.7ex]
\end{tabular}
\caption{Classification of the isomorphism types of DTS quasigroups
of order~13 according to their automorphism group.}
\label{tbl:automorphisms}
\end{table}

Using the sizes of the automorphism groups from
Table~\ref{tbl:automorphisms}, we can easily compute the total number
of DTS quasigroups of order~13 by taking the sum of $13!/\lvert\mathrm{Aut}(Q)\rvert$
over all isomorphism types~$Q$, which comes out to $7\,502\,250\,290\,008\,320$.

If we attempt to generate the DTS quasigroups of orders 15, 19 or~21, Mace4
instantly produces plenty of models and soon runs out of memory. For the
remaining orders, the program tends to produce fewer results. Using the
above input, we were not able to obtain DTS quasigroups of even orders
greater than~18 in a reasonable amount of time, but we did obtain ones of
orders 25, 27, 31 and~37.

To determine the existence spectrum of LDTS in~\cite{ldts} we needed to
obtain LDTS of certain orders, which were as high as~40. We did this by
prescribing a suitable automorphism as part of the input to Mace4.
Generally Mace4 can then produce a model within a few seconds, but
the time varies greatly. To date, the largest model that we have been
able to obtain this way is a pure DTS quasigroup of order~58 with an
automorphism of type~$29^2$. However this technique is not always
successful. For example, we were not able to generate a pure flexible
DTS quasigroup of order~16, instead it was generated using the program
Paradox~\cite{CS} which found an automorphism-free model.

\appendix
\section*{Appendix. Examples of DTS quasigroups of order~13}
\renewcommand\thesection{A}
\setcounter{thm}{0}
It is clearly impossible to list all DTS quasigroups of order~13 but below
are given some which may be of particular interest. These are the unique
proper system with all middle vertices of degree~3, all six systems with
two surface constituents of genus~1, all systems having an automorphism
group of order greater than or equal to~4, and at least one example of a
system having just one surface constituent of genus 0, 1, 2, 3 or~4,
respectively.

In the following examples let
$X=\{\mathtt 0,\mathtt 1,\mathtt 2,\mathtt 3,\mathtt 4,\mathtt 5,\mathtt 6,\mathtt 7,\mathtt 8,\mathtt 9,\mathtt T,\mathtt E,\mathtt W\}$.
For simplicity, we omit commas from the triples.
\begin{ex}\label{ldts13deg3}
Define
$\mathcal T = \{\str{018}$, $\str{09E}$, $\str{0TW}$, $\str{19T}$,
$\str{1EW}$, $\str{259}$, $\str{268}$, $\str{27W}$, $\str{2TE}$,
$\str{35T}$, $\str{36E}$, $\str{379}$, $\str{38W}$, $\str{45W}$,
$\str{46T}$, $\str{47E}$, $\str{489}$, $\str{58E}$, $\str{69W}$,
$\str{78T}\}$,\\
$\mathcal C_1 = \{\dtr{203}$, $\dtr{304}$, $\dtr{402}$, $\dtr{214}$, $\dtr{413}$, $\dtr{312}\}$ and\\
$\mathcal C_2 = \{\dtr{506}$, $\dtr{607}$, $\dtr{705}$, $\dtr{517}$, $\dtr{716}$, $\dtr{615}\}$.\\
Then $\mathcal C_1$ and $\mathcal C_2$ are surface constituents of genus~0, and
$(X,\mathcal T \cup \mathcal C_1 \cup \mathcal C_2)$ is the unique
proper LDTS(13) such that all middle vertices are of degree~3.
The system is automorphism-free and flexible.
\end{ex}

\begin{ex}\label{ldts13g11}
The 6 systems with two surface constituents of genus~1 are defined as follows.
\begin{enumerate}
\item Define $\mathcal T = \{\str{09E}$, $\str{38W}$, $\str{48E}\}$,\\
$\mathcal C_1 = \{\dtr{102}$, $\dtr{203}$, $\dtr{304}$, $\dtr{405}$, $\dtr{501}$, $\dtr{164}$, $\dtr{463}$, $\dtr{365}$,
$\dtr{56W}$, $\dtr{W62}$, $\dtr{261}$, $\dtr{19W}$, $\dtr{W95}$, $\dtr{594}$, $\dtr{491}$, $\dtr{1E5}$, $\dtr{5E3}$, $\dtr{3E2}$,
$\dtr{2EW}$, $\dtr{WE1}\}$ and\\
$\mathcal C_2 = \{\dtr{607}$, $\dtr{70W}$, $\dtr{W0T}$, $\dtr{T08}$, $\dtr{806}$, $\dtr{317}$, $\dtr{718}$, $\dtr{81T}$,
$\dtr{T13}$, $\dtr{24T}$, $\dtr{T4W}$, $\dtr{W47}$, $\dtr{742}$, $\dtr{258}$, $\dtr{857}$, $\dtr{75T}$, $\dtr{T52}$, $\dtr{297}$,
$\dtr{793}$, $\dtr{39T}$, $\dtr{T96}$, $\dtr{698}$, $\dtr{892}$, $\dtr{6ET}$, $\dtr{TE7}$, $\dtr{7E6}\}$.\\
Then $\mathcal C_1$ and $\mathcal C_2$ are surface constituents of genus~1, and
$(X,\mathcal T \cup \mathcal C_1 \cup \mathcal C_2)$,
$(X,\mathcal T \cup \mathcal C_1\op \cup \mathcal C_2)$,
$(X,\mathcal T \cup \mathcal C_1 \cup \mathcal C_2\op)$ and
$(X,\mathcal T \cup \mathcal C_1\op \cup \mathcal C_2\op)$
are non-flexible, automorphism-free LDTS(13)s.

\item Define $\mathcal T = \{\str{09E}$, $\str{137}$, $\str{679}\}$,\\
$\mathcal C_1 = \{\dtr{102}$, $\dtr{203}$, $\dtr{304}$, $\dtr{405}$, $\dtr{501}$, $\dtr{16W}$, $\dtr{W64}$, $\dtr{463}$,
$\dtr{365}$, $\dtr{562}$, $\dtr{261}$, $\dtr{295}$, $\dtr{594}$, $\dtr{49W}$, $\dtr{W92}$, $\dtr{1E5}$, $\dtr{5E3}$, $\dtr{3E2}$, $\dtr{2EW}$,
$\dtr{WE1}\}$ and\\
$\mathcal C_2 = \{\dtr{60T}$, $\dtr{T07}$, $\dtr{70W}$, $\dtr{W08}$, $\dtr{806}$, $\dtr{418}$, $\dtr{819}$, $\dtr{91T}$,
$\dtr{T14}$, $\dtr{427}$, $\dtr{72T}$, $\dtr{T28}$, $\dtr{824}$, $\dtr{83W}$, $\dtr{W3T}$, $\dtr{T39}$, $\dtr{938}$, $\dtr{758}$, $\dtr{85T}$,
$\dtr{T5W}$, $\dtr{W57}$, $\dtr{4ET}$, $\dtr{TE6}$, $\dtr{6E8}$, $\dtr{8E7}$, $\dtr{7E4}\}$.\\
Then $\mathcal C_1$ and $\mathcal C_2$ are surface constituents of genus~1, and
$(X,\mathcal T \cup \mathcal C_1 \cup \mathcal C_2)$ and
$(X,\mathcal T \cup \mathcal C_1 \cup \mathcal C_2\op)$
are non-flexible, automorphism-free LDTS(13)s.
\end{enumerate}
\end{ex}

\begin{ex}\label{cyclicsts13}
The DTS quasigroup that has automorphism group of order~39 is the
Steiner quasigroup which comes from the cyclic STS(13) obtained
from the starter blocks $\str{014}$, $\str{027}$ under the action
of the permutation\\ $(\mathtt 0,\mathtt 1,\mathtt 2,\mathtt
3,\mathtt 4,\mathtt 5,\mathtt 6,\mathtt 7,\mathtt 8,\mathtt
9,\mathtt T,\mathtt E,\mathtt W)$.
\end{ex}

\begin{ex}\label{ldts13C13}
The 2 DTS quasigroups that have automorphism group $C_{13}$ are defined by the
triples obtained from the following starter blocks under the action of the
permutation
$(\mathtt 0,\mathtt 1,\mathtt 2,\mathtt 3,\mathtt 4,\mathtt 5,\mathtt 6,\mathtt 7,\mathtt 8,\mathtt 9,\mathtt T,\mathtt E,\mathtt W)$.
The starter blocks for $\mathcal C$ are $\dtr{105}$, $\dtr{507}$, $\dtr{703}$, $\dtr{301}$.
Then $\mathcal C$ is a surface constituent of genus~1, and $(X,\mathcal C)$ and $(X,\mathcal C\op)$ are pure, non-flexible LDTS(13)s.
\end{ex}

\begin{ex}\label{ldts13D12}
The 2 DTS quasigroups that have automorphism group $D_{12}$ of order~12 are defined by the
triples obtained from the following starter blocks under the action of the
group generated by the permutations
$(\mathtt 0,\mathtt 1,\mathtt 2,\mathtt 3,\mathtt 4,\mathtt 5)(\mathtt 6,\mathtt 7,\mathtt 8,\mathtt 9,\mathtt T,\mathtt E)$ and
$(\mathtt 0,\mathtt 5)(\mathtt 1,\mathtt 4)(\mathtt 2,\mathtt 3)(\mathtt 6,\mathtt 8)(\mathtt 9,\mathtt E)$.
The starter blocks for $\mathcal T$ are $\str{018}$, $\str{024}$, $\str{03W}$, $\str{69W}$, and
for $\mathcal C$ are $\dtr{60E}$, $\dtr{E09}$.
Then $\mathcal C$ is a surface constituent of genus~1, and
$(X, \mathcal T \cup \mathcal C)$ and $(X, \mathcal T \cup \mathcal C\op)$ are non-flexible LDTS(13)s.
\end{ex}

\begin{ex}\label{ldts13D10}
The unique DTS quasigroup that has automorphism group $D_{10}$ of order~10
is defined by the triples obtained from the following starter blocks under the
action of the group generated by the permutations
$(\mathtt 0,\mathtt 1,\mathtt 2,\mathtt 3,\mathtt 4)(\mathtt 5,\mathtt 6,\mathtt 7,\mathtt 8,\mathtt 9)$ and
$(\mathtt 1,\mathtt 4)(\mathtt 2,\mathtt 3)(\mathtt 6,\mathtt 9)(\mathtt 7,\mathtt 8)(\mathtt T,\mathtt E)$.
The starter block
for $\mathcal C_1$ is $\dtr{0T4}$,
for $\mathcal C_2$ is $\dtr{5T8}$, and
for $\mathcal T$ are $\str{026}$, $\str{05W}$, $\str{078}$, $\str{TEW}$.
Then $\mathcal C_1$ and $\mathcal C_2$ are surface constituents of genus~0, and
$(X, \mathcal T \cup \mathcal C_1 \cup \mathcal C_2)$ is a flexible LDTS(13).
\end{ex}

\begin{ex}\label{ldts13C10}
The 2 DTS quasigroups that have automorphism group $C_{10}$ can both be
obtained from the starter blocks $\str{017}$, $\str{05W}$, $\str{TEW}$,
$\dtr{0T2}$, $\dtr{2E0}$.
The first LDTS is defined by the triples obtained from the starter blocks
under the action of the permutation
$(\mathtt 0,\mathtt 1,\mathtt 2,\mathtt 3,\mathtt 4,\mathtt 5,\mathtt 6,\mathtt 7,\mathtt 8,\mathtt 9)(\mathtt T,\mathtt E)$.
The second LDTS is defined by the triples obtained from the starter blocks
under the action of the permutation
$(\mathtt 0,\mathtt 1,\mathtt 2,\mathtt 3,\mathtt 4,\mathtt 5,\mathtt 6,\mathtt 7,\mathtt 8,\mathtt 9)$.
Both LDTS(13)s are flexible and each consists of 2 surface constituents of genus~0.
\end{ex}

\begin{ex}\label{ldts13C6}
The 5 DTS quasigroups that have automorphism group $C_6$ are defined by the
triples obtained from the following starter blocks under the action of the
permutation
$(\mathtt 0,\mathtt 1,\mathtt 2,\mathtt 3,\mathtt 4,\mathtt 5)(\mathtt 6,\mathtt 7,\mathtt 8,\mathtt 9,\mathtt T,\mathtt E)$.
\begin{enumerate}
  \item The starter blocks for $\mathcal D$ are $\str{06W}$, $\str{68T}$, $\dtr{104}$, $\dtr{407}$, $\dtr{708}$, $\dtr{80E}$, $\dtr{E0T}$, $\dtr{T01}$. Then $(X,\mathcal D)$ and $(X,\mathcal D\op)$ are flexible LDTS(13)s, each consisting of 3 surface constituents of genus~0.

  \item The starter blocks for $\mathcal D$ are $\str{06W}$, $\str{09T}$, $\str{68T}$, $\dtr{104}$, $\dtr{40E}$, $\dtr{E08}$, $\dtr{801}$. Then $(X,\mathcal D)$ and $(X,\mathcal D\op)$ are flexible LDTS(13)s, each consisting of 3 surface constituents of genus~0.

  \item The starter blocks for $\mathcal D$ are $\str{03W}$, $\str{68T}$, $\str{69W}$, $\dtr{106}$, $\dtr{607}$, $\dtr{702}$, $\dtr{20T}$, $\dtr{T09}$, $\dtr{901}$. Then $(X,\mathcal D)$ is a non-flexible LDTS(13) consisting of a single surface constituent of genus~1.
\end{enumerate}
\end{ex}

\begin{ex}\label{ldts13S3}
The 7 DTS quasigroups that have automorphism group $S_3$ are defined by the
triples obtained from the following starter blocks under the action of the
group generated by the permutations
$(\mathtt 0,\mathtt 1,\mathtt 2)(\mathtt 3,\mathtt 4,\mathtt 5)(\mathtt 6,\mathtt 7,\mathtt 8)(\mathtt 9,\mathtt T,\mathtt E)$ and
$(\mathtt 0,\mathtt 3)(\mathtt 1,\mathtt 5)(\mathtt 2,\mathtt 4)(\mathtt 7,\mathtt 8)(\mathtt T,\mathtt E)$.
\begin{enumerate}
\item The starter blocks for $\mathcal T$ are $\str{678}$, $\str{69W}$, $\str{9TE}$,
for $\mathcal C_1$ are $\dtr{061}$, $\dtr{163}$, $\dtr{0W2}$, and
for $\mathcal C_2$ are $\dtr{094}$, $\dtr{491}$, $\dtr{198}$, $\dtr{893}$.
Then $\mathcal C_1$ is a surface constituent of genus~0,
$\mathcal C_2$ is a surface constituent of genus~1, and
$(X,\mathcal T \cup \mathcal C_1 \cup \mathcal C_2)$ and $(X,\mathcal T \cup \mathcal C_1\op \cup \mathcal C_2)$ are non-flexible LDTS(13)s.

\item The starter blocks for $\mathcal T$ are $\str{05W}$, $\str{678}$, $\str{69W}$, $\str{9TE}$, and
for $\mathcal C$ are $\dtr{061}$, $\dtr{164}$, $\dtr{46E}$, $\dtr{E60}$, $\dtr{092}$, $\dtr{293}$.
Then $\mathcal C$ is a surface constituent of genus~1, and
$(X,\mathcal T \cup \mathcal C)$ and $(X,\mathcal T \cup \mathcal C\op)$ are non-flexible LDTS(13)s.

\item The starter blocks for $\mathcal T$ are $\str{016}$, $\str{05W}$, $\str{678}$, $\str{69W}$, $\str{9TE}$, and
for $\mathcal C$ are $\dtr{094}$, $\dtr{491}$, $\dtr{198}$, $\dtr{893}$.
Then $\mathcal C$ is a surface constituent of genus~1, and
$(X,\mathcal T \cup \mathcal C)$ is a non-flexible LDTS(13).

\item The starter blocks for $\mathcal T$ are $\str{039}$, $\str{04E}$, $\str{057}$, $\str{678}$, $\str{69W}$, $\str{9TE}$, and
for $\mathcal C$ are $\dtr{061}$, $\dtr{16E}$, $\dtr{E63}$, $\dtr{0W2}$.
Then $\mathcal C$ is a surface constituent of genus~0, and
$(X,\mathcal T \cup \mathcal C)$ is a non-flexible LDTS(13).

\item The starter blocks for $\mathcal T$ are $\str{017}$, $\str{039}$, $\str{04E}$, $\str{05W}$, $\str{08T}$, $\str{678}$, $\str{69W}$, $\str{9TE}$.
Then $(X,\mathcal T)$ is the non-cyclic STS(13).
\end{enumerate}
\end{ex}

\begin{ex}\label{ldts13C5}
The 8 DTS quasigroups that have automorphism group $C_5$ are defined by the
triples obtained from the following starter blocks under the action of the
permutation
$(\mathtt 0,\mathtt 1,\mathtt 2,\mathtt 3,\mathtt 4)(\mathtt 5,\mathtt 6,\mathtt 7,\mathtt 8,\mathtt 9)$.
The starter blocks for $\mathcal C_0$ are $\dtr{0T1}$, $\dtr{1E0}$,
for $\mathcal C_1$ are $\dtr{706}$, $\dtr{609}$, $\dtr{90W}$, $\dtr{W07}$, $\dtr{5T7}$, $\dtr{5E6}$,
for $\mathcal C_2$ are $\dtr{706}$, $\dtr{608}$, $\dtr{803}$, $\dtr{307}$, $\dtr{0W5}$, $\dtr{5W1}$, $\dtr{5T6}$, $\dtr{5E8}$,
for $\mathcal T_1$ are $\str{025}$, $\str{TEW}$, and
$\mathcal T_2 = \{\str{TEW}\}$.
Then $\mathcal C_0$ is a surface constituent of genus~0, $\mathcal C_1$ and
$\mathcal C_2$ are surface constituents of genus~2, and
$(X,\mathcal T_1 \cup \mathcal C_0 \cup \mathcal C_1)$,
$(X,\mathcal T_1 \cup \mathcal C_0 \cup \mathcal C_1\op)$,
$(X,\mathcal T_1 \cup \mathcal C_0\op \cup \mathcal C_1)$,
$(X,\mathcal T_1 \cup \mathcal C_0\op \cup \mathcal C_1\op)$,
$(X,\mathcal T_2 \cup \mathcal C_0 \cup \mathcal C_2)$,
$(X,\mathcal T_2 \cup \mathcal C_0 \cup \mathcal C_2\op)$,
$(X,\mathcal T_2 \cup \mathcal C_0\op \cup \mathcal C_2)$ and
$(X,\mathcal T_2 \cup \mathcal C_0\op \cup \mathcal C_2\op)$
are non-flexible LDTS(13)s.
\end{ex}

\begin{ex}\label{ldts13C22}
The 17 DTS quasigroups that have automorphism group $C_2\times C_2$ are defined
by the triples obtained from the following starter blocks under the action of
the group generated by the permutations
$(\mathtt 0,\mathtt 1)(\mathtt 2,\mathtt 3)(\mathtt 4,\mathtt 5)(\mathtt 6,\mathtt 7)(\mathtt 8,\mathtt 9)$ and
$(\mathtt 0,\mathtt 9)(\mathtt 1,\mathtt 8)(\mathtt 2,\mathtt 7)(\mathtt 3,\mathtt 6)(\mathtt 4,\mathtt 5)(\mathtt T,\mathtt E)$.
\begin{enumerate}
\item
The starter blocks for $\mathcal T$ are $\str{048}$, $\str{45W}$, $\str{TEW}$,
for $\mathcal D_0$ are $\dtr{0T5}$, $\dtr{5T8}$, $\dtr{8T6}$, $\dtr{6T2}$, $\dtr{2T0}$,
for $\mathcal C_1$ are $\dtr{03W}$, $\dtr{W38}$, $\dtr{839}$, $\dtr{930}$, and
for $\mathcal C_2$ are $\dtr{243}$, $\dtr{346}$.
Then $\mathcal C_1$ and $\mathcal C_2$ are surface constituents of genus~0, $\mathcal D_0$ consists of 2 surface constituents of genus~0, and
$(X,\mathcal T \cup \mathcal D_0 \cup \mathcal C_1 \cup \mathcal C_2)$,
$(X,\mathcal T \cup \mathcal D_0 \cup \mathcal C_1\op \cup \mathcal C_2)$,
$(X,\mathcal T \cup \mathcal D_0 \cup \mathcal C_1 \cup \mathcal C_2\op)$ and
$(X,\mathcal T \cup \mathcal D_0 \cup \mathcal C_1\op \cup \mathcal C_2\op)$
are non-flexible LDTS(13)s.

\item
The starter blocks for $\mathcal T$ are $\str{01W}$, $\str{048}$, $\str{26W}$, $\str{45W}$, $\str{TEW}$,
for $\mathcal C_1$ are $\dtr{305}$, $\dtr{507}$, $\dtr{706}$, $\dtr{603}$, and
for $\mathcal C_2$ are $\dtr{0T2}$, $\dtr{2T5}$, $\dtr{5T6}$, $\dtr{6T8}$, $\dtr{8T1}$.
Then $\mathcal C_1$ and $\mathcal C_2$ are surface constituents of genus~0, and
$(X,\mathcal T \cup \mathcal C_1 \cup \mathcal C_2)$ and
$(X,\mathcal T \cup \mathcal C_1\op \cup \mathcal C_2)$
are non-flexible LDTS(13)s.

\item
The starter blocks for $\mathcal T$ are $\str{01W}$, $\str{048}$, $\str{27W}$, $\str{45W}$, $\str{TEW}$, and
for $\mathcal C$ are $\dtr{305}$, $\dtr{506}$, $\dtr{607}$, $\dtr{703}$, $\dtr{0T2}$, $\dtr{2T5}$, $\dtr{5T7}$, $\dtr{7T9}$, $\dtr{9T0}$.
Then $\mathcal C$ is a surface constituent of genus~0, and
$(X,\mathcal T \cup \mathcal C)$ and
$(X,\mathcal T \cup \mathcal C\op)$
are non-flexible LDTS(13)s.

\item
The starter blocks for $\mathcal T$ are $\str{01E}$, $\str{048}$, $\str{07T}$, $\str{09W}$, $\str{23W}$, $\str{45W}$, $\str{TEW}$, and
for $\mathcal C$ are $\dtr{206}$, $\dtr{603}$, $\dtr{305}$, $\dtr{502}$, $\dtr{2T5}$, $\dtr{5T3}$.
Then $\mathcal C$ is a surface constituent of genus~0, and
$(X,\mathcal T \cup \mathcal C)$ and
$(X,\mathcal T \cup \mathcal C\op)$
are non-flexible LDTS(13)s.

\item
The starter blocks for $\mathcal T$ are $\str{01T}$, $\str{02E}$, $\str{048}$, $\str{09W}$, $\str{25T}$, $\str{26W}$, $\str{45W}$, $\str{TEW}$, and
for $\mathcal C$ are $\dtr{305}$, $\dtr{507}$, $\dtr{706}$, $\dtr{603}$.
Then $\mathcal C$ is a surface constituent of genus~0, and
$(X,\mathcal T \cup \mathcal C)$ and
$(X,\mathcal T \cup \mathcal C\op)$
are non-flexible LDTS(13)s.

\item
The starter blocks for $\mathcal T$ are $\str{01W}$, $\str{048}$, $\str{23W}$, $\str{45W}$, $\str{TEW}$, and
for $\mathcal C$ are $\dtr{206}$, $\dtr{603}$, $\dtr{305}$, $\dtr{502}$, $\dtr{0T9}$, $\dtr{9T2}$, $\dtr{2T5}$, $\dtr{5T7}$, $\dtr{7T0}$.
Then $\mathcal C$ is a surface constituent of genus~1, and
$(X,\mathcal T \cup \mathcal C)$ and
$(X,\mathcal T \cup \mathcal C\op)$
are non-flexible LDTS(13)s.

\item
The starter blocks for $\mathcal T$ are $\str{01T}$, $\str{048}$, $\str{09W}$, $\str{25T}$, $\str{27W}$, $\str{45W}$, $\str{TEW}$, and
for $\mathcal C$ are $\dtr{20E}$, $\dtr{E03}$, $\dtr{305}$, $\dtr{507}$, $\dtr{706}$, $\dtr{602}$.
Then $\mathcal C$ is a surface constituent of genus~1, and
$(X,\mathcal T \cup \mathcal C)$
is a non-flexible LDTS(13).

\item
The starter blocks for $\mathcal T$ are $\str{01E}$, $\str{048}$, $\str{09W}$, $\str{26W}$, $\str{45W}$, $\str{TEW}$, and
for $\mathcal C$ are $\dtr{203}$, $\dtr{305}$, $\dtr{506}$, $\dtr{60T}$, $\dtr{T07}$, $\dtr{702}$, $\dtr{2T5}$, $\dtr{5T3}$.
Then $\mathcal C$ is a surface constituent of genus~2, and
$(X,\mathcal T \cup \mathcal C)$ and
$(X,\mathcal T \cup \mathcal C\op)$
are non-flexible LDTS(13)s.
\end{enumerate}
\end{ex}

\begin{ex}\label{ldts13g3}
The system is defined by the triples obtained from the following starter blocks under the action of the permutation
$(\mathtt 0,\mathtt 1,\mathtt 2)(\mathtt 3,\mathtt 4,\mathtt 5)(\mathtt 6,\mathtt 7,\mathtt 8)(\mathtt 9,\mathtt T,\mathtt E)$.
The starter blocks for $\mathcal C$ are $\dtr{16E}$, $\dtr{E63}$, $\dtr{368}$, $\dtr{862}$, $\dtr{265}$, $\dtr{561}$, $\dtr{096}$, $\dtr{69W}$, $\dtr{W98}$, $\dtr{893}$, $\dtr{391}$, $\dtr{195}$, $\dtr{59T}$, $\dtr{T90}$, $\dtr{3W0}$, $\dtr{0W5}$, and
$\mathcal T = \{\str{012}$, $\str{345}\}$.
Then $\mathcal C$ is a surface constituent of genus~3, and
$(X, \mathcal T \cup \mathcal C)$ is a non-flexible LDTS(13).
The automorphism group of the DTS quasigroup is $C_3$.
\end{ex}

\begin{ex}\label{ldts13g4}
The system is defined by the triples obtained from the following starter blocks under the action of the permutation
$(\mathtt 0,\mathtt 1)(\mathtt 2,\mathtt 3)(\mathtt 4,\mathtt 5)(\mathtt 6,\mathtt 7)(\mathtt 8,\mathtt 9)(\mathtt T,\mathtt E)$.
The starter blocks for $\mathcal{C}$ are $\dtr{065}$, $\dtr{162}$, $\dtr{260}$, $\dtr{561}$, $\dtr{084}$, $\dtr{287}$, $\dtr{382}$, $\dtr{483}$, $\dtr{580}$, $\dtr{785}$, $\dtr{0T9}$, $\dtr{1T0}$, $\dtr{2T1}$, $\dtr{3T4}$, $\dtr{4T5}$, $\dtr{5T3}$, $\dtr{6T7}$, $\dtr{7T2}$, $\dtr{8T6}$, $\dtr{9T8}$, $\dtr{0W2}$, $\dtr{2W4}$, $\dtr{4W6}$, $\dtr{6W8}$, $\dtr{8W1}$, and
$\mathcal T = \{\str{TEW}\}$.
Then $\mathcal C$ is a surface constituent of genus~4, and
$(X, \mathcal T \cup \mathcal C)$ is a non-flexible LDTS(13).
The automorphism group of the DTS quasigroup is $C_2$.
\end{ex}

\end{document}